\documentclass[12pt]{amsart}
\usepackage{amssymb}
\usepackage{amsmath}
\usepackage{amsfonts}
\def \and{\qquad\text{and}\qquad}

\newtheorem{proposition}{Proposition}[section]
\newtheorem{theorem}[proposition]{Theorem}

\newtheorem{lemma}[proposition]{Lemma}
\theoremstyle{definition}

\theoremstyle{remark}
\newtheorem{remark}[proposition]{Remark}

\numberwithin{equation}{section}

\begin{document}

\title[On Lagrange's discrete model of the wave equationin space dimension greater than one]{Lagrange's discrete model of the wave equation in space dimension greater than one}

\author{Massimo Villarini}
\footnote[1]{Massimo Villarini, Dipartimento di Scienze Fisiche, Informatiche e Matematiche, via Campi 213/b 41100, Universit\'a di Modena e Reggio Emilia, Modena, Italy

E-mail: massimo.villarini@unimore.it}

\begin{abstract}
A celebrated theorem of Lagrange \cite{lagrange}  concerns a discrete model, based on Newton's Second Law of dynamics, for the PDE describing the transversal motion of a string: such mechanical  model is expressed through a family of second order ODEs, depending on a discretization parameter, whose solutions Lagrange proved to converge uniformly to the solutions of the PDE, as the discretization parameter tends to $0$.  Answering to a question posed by Gallavotti \cite{gallavotti}, we generalize this theorem to the case of $n$-dimensional space variable, $n>1$. The proof is based on the convergence analysis of the simplest finite difference numerical scheme for the wave equation.
\end{abstract}

\maketitle

\section{Introduction}
In 1759 Lagrange published two memoires on sound propagation \cite{lagrange} related to  a scientific controversy between Euler and d'Alembert on their solutions of the PDE modelling the vibration of a string.

In those memoires Lagrange proposed a  mechanical model of a discretized vibrating string, based on Newton's Second Law of dynamics, leading to a linear second order  ODE with $N$ degrees of freedom, $N$ being the number of points of a lattice discretizing the string continuum. He succeded in diagonalizing the matrix of the linear ODE,  therefore  finding its explicit analytic solution: this result seems to be specific to the spatial one-dimensional case of the vibrating string  (see \cite{gallavotti} \textsection 4, comment before Proposition 16). Letting $N \rightarrow \infty$, {\it i.e.} passing to the continuum limit, he proved convergence of the explicit solutions of the ODE to solutions of the vibrating string PDE, hence  reobtaining and generalizing Euler results: we will quote this result as {\it Lagrange's 1759 Theorem}.  This was perhaps the first rigorous example of reduction of a scalar field PDE to first principles. A complete account of Lagrange's theory can be found in the book \cite{gallavotti} by G. Gallavotti, where the author also generalizes Lagrange's mechanical and ODE models to the case of spatial dimension $n>1$, and proposes the natural generalization of Lagrange's 1759 Theorem to that case (see Proposition 16. \textsection $4$ in \cite{gallavotti} and the following Observation, where the sought generalization is settled as an open question).

The main result of the present article, Theorem (\ref{principale}),  is  the generalization of Lagrange's 1759 Theorem to the case of spatial dimension $n>1$: it answers in the affirmative to the question posed implicitly by Gallavotti. The proof is based on the convergence analysis of the simplest finite difference scheme for the wave equation, {\it cfr.} \cite{john-adv}. The same approach, namely reducing, through convergence analysis of suitable finite difference scheme, the  PDE modeling scalar field theory to a family of ODEs depending on a discretizing parameter,  turns to be useful for the heat equation, too,  and will be subject of a forthcoming article.

Theorem (\ref{principale}) is stated in the next section, where it is proved in a simplified version. More general versions are proved in the subsequent sections, and finally the theorem is proved in  full generality. 

\section{The homogeneous case}
In this section we will state our main result, generalizing Lagrange's 1759 Theorem to the case of spatial dimension $n>1$, and we will prove it in a simplified form. Let $\Omega$ be an open, connected, bounded subset of $\mathbb{R}^n$, $n \geq 1$, having boundary $\partial \Omega$ which is a smooth $(n-1)$-manifold: weakening of these hypotheses is briefly discussed at the end of the article. The d'Alembert (continuos) differential operator is
\[
\Box = \frac{\partial^2}{\partial t^2} - c^2 \Delta_x
\]
where
\[
\Delta_x = \frac{\partial^2}{\partial x_1^2} + \cdots +\frac{\partial^2}{\partial x_n^2} .
\]
Up to a linear change of coordinates d'Alembert operator reduces to
\[
\Box = \frac{\partial^2}{\partial t^2} -  \Delta_x
\]
and in this form it will be considered throughout this article, without explicit displaying in the Laplace operator the space variable $x$. The scalar variable $t$ will be referred to as time.

 Let $T>0$. We will be interested in the following mixed initial/boundary problem for the wave equation
\[
(L)
\begin{cases}
\frac{\partial^2 u}{\partial t^2} - a(x) \Delta u + \sigma (x) u = w (x,t)  \; , \; in \; \Omega \times (-T , T) \\
u(x,0)=f(x) \;, \; in \; \Omega \times \{0 \} \\
\frac{\partial u}{\partial t}(x,0) = g(x) \; , \; \Omega \times \{0 \} \\
u(x,t) \equiv h(x) \; , \; in \; \partial \Omega \times (-T , T)
\end{cases}
\]
The choice of constant $T$ is instrumental in the proof, and  it does not correspond to any restriction of the domain of definition of the solution of $(L)$.
The functions
\[
a, \sigma : \Omega \rightarrow \mathbb{R}
\]
are the coefficients of the PDE  ($a>0$ is the velocity of wave propagation, $\sigma$ is the flexibility coefficent measuring resistence to deformation of the medium  ), the functions
\[
f, g : \Omega \rightarrow \mathbb{R}
\]
are the initial data
\[
h  : \partial \Omega \rightarrow \mathbb{R}
\]
is the boundary datum and
\[
w  : \Omega  \times \mathbb{R} \rightarrow \mathbb{R}
\]
is the forcing term. We suppose that initial and boundary data satisfy the {\it compatibility conditions}
\[
(LC)
\begin{cases}
lim_{y \rightarrow x} f(y) = h(x) \; y \in \Omega \;, \; x \in \partial \Omega\\
lim_{y \rightarrow x} g(y)=0 \; y \in \Omega \;, \; x \in \partial \Omega \\
\Delta_{\vert \partial \Omega} h =0
\end{cases}
\]
where $\Delta_{\vert \partial \Omega}$ is the restriction of the Laplace operator with respect to $x$ to the submanifold $\partial \Omega$. In this section we will mainly consider the following simplified case of $(L)$
\[
(E)
\begin{cases}
\Box u = 0  \; , \; in \; \Omega \times (-T , T) \\
u(x,0)=f(x) \;, \; in \; \Omega \times \{0 \} \\
\frac{\partial u}{\partial t}(x,0) = g(x) \; , \; \Omega \times \{0 \}
\end{cases}
\]
{\it i.e.} the homogeneous, constant velocity ($a=1$), infinitely flexible ($\sigma =0$), initial data case of $(L)$.

We introduce now a discretization of $(E)$. Let $\Delta x$, $\Delta t$ two positive numbers, and let
\[
\begin{cases}
\Sigma_{\Delta x} = \{ x \in \mathbb{R}^n : x= k \Delta x, \; k \in \mathbb{Z}^n  \} \\
\Sigma_{\Delta t} = \{ t \in \mathbb{R} : t= p \Delta x, \; p \in \mathbb{Z}  \} \\
\Sigma_{\Delta x , \Delta t}= \Sigma_{\Delta x} \times \Sigma_{\Delta t}
\end{cases}
\]
be the corresponding lattices. We define the set $\mathcal{A}_T$ of {\it admissible lattices} as
\[
\mathcal{A}_T = \{ (\Delta x, \Delta t): \frac{T}{\Delta t} \in \mathbb{N}, \; \frac{\Delta t}{\Delta x} \leq \frac{1}{\sqrt{n}} \}
\]
The condition appearing in the definition of $\mathcal{A}_T$
\begin{equation}\label{cfl}
\frac{\Delta t}{\Delta x} \leq \frac{1}{\sqrt{n}}
\end{equation}
is the celebrated {Courant-Friedrichs-Lewy condition}, \cite{cfl}.  Let
\[
E_{\infty} = \cup_{(\Delta x , \Delta t) \in \mathcal{A}_T} \Sigma_{(\Delta x , \Delta t)} \cap (\Omega \times (-T , T))
\]
$E_{\infty}$ is dense in $\Omega \times (-T , T)$.

We will consider the discrete differential operators
\begin{itemize}
\item [] $\delta_{\Delta t} u (x,t) =\frac{u(x,t+\Delta t) - u(x,t)}{\Delta t}$ \\
\item [] $\delta_{\Delta t}^{-1} u (x,t) = \frac{u(x,t)-u(x,t-\Delta t)}{\Delta t}$ \\
\item [] $\delta_{\overline {\Delta t}} u (x,t) = \frac{u(x,t+\Delta t) - u(x,t- \Delta t)}{2 \Delta t}$ \\
\item [] $\delta_{\Delta t}^{-1} \circ \delta_ {\Delta t} u (x,t) = \frac{u(x,t+\Delta t) -  2 u(x,t) + u(x,t- \Delta t)}{ \Delta t^2}$
\end{itemize}
and the analogously defined discrete differential operators $\delta_{k,\Delta x}, \delta_{k , \Delta x}^{-1}, \delta_{k , \Delta x}^{-1} \circ \delta_ {k , \Delta x} $, where $k=1, \ldots , n$ and {\it e.g.}
\[
\delta_{k,\Delta x} = \frac{u(x+\Delta x \;  e_k , t) - u(x,t)}{\Delta x }
\]
where $e_k$ is the $k$-th vector of the canonical basis in $\mathbb{R}^n$. Let
\[
\Sigma_{\Delta x} \cap \Omega = \Omega_{\Delta x} \cup \partial_{\Delta x}
\]
where $x \in \Omega_{\Delta x}$ if it is a point of the lattice $\Sigma_{\Delta x}$ which belongs to $\Omega$  togheter with all its $2n$-neighbours, while $x \in \partial_{\Delta x} \Omega$ is a point of the lattice which belongs to $\overline \Omega$, the closure of $\Omega$, such that at least one of its $2n$-neighbours does not belong to $\Omega$. We define the discrete differential operator ($(\Delta x , \Delta t)$-approximation of the d'Alembert operator)
\[
\Box_{\Delta x, \Delta t} = \delta_{\Delta t}^{-1} \circ \delta_ {\Delta t} - \sum_{k=1}^n  \delta_{k , \Delta x}^{-1} \circ \delta_ {k , \Delta x} 
\]
and we will consider the following approximated version of $(E)$
\[
(E)_{\Delta x , \Delta t}
\begin{cases}
\Box_{\Delta x, \Delta t} v(x,t)=0 \; in \; \Sigma_{\Delta x, \Delta t} \cap (\Omega \times (-T,T)) \\
v(x,0)= f(x) \; in \; (\Sigma_{\Delta x} \cap \Omega) \times \{ 0\} \\
\delta_{\overline {\Delta t}} v(x,0)=g(x) \; in \; (\Sigma_{\Delta x} \cap \Omega) \times \{ 0\} 
\end{cases}
\]
After considering this initial value problem, we will add to it the boundary condition
\[
v(x,t)= 0 \; in \; \partial \Omega_{\Delta x } \times ((-T,T)) \cap \Sigma_{\Delta t}
\]
hence getting an approximated version $(L)_{\Delta x , \Delta t}$ of a simplified form of $(L)$. The approximated versions of initial and mixed valued problems for a PDE are referred to as {\it finite difference numerical schemes} in numerical analysis. In particular the numerical scheme based on the definition of $\Box_{\Delta x , \Delta t}$ is defined the {\it simplest numerical scheme} for the wave equation in \cite{john-adv} \textsection {7}. We observe that $(E)_{\Delta x , \Delta t}$, being a linear explicit scheme, has always a solution, which is unique.

We will use a standard notation in the theory of evolution equations: a function
\[
u : \Omega \times \mathbb{R} \rightarrow \mathbb{R} \;, \; (x,t) \rightarrow u(x,t)
\]
will be also denoted as
\[
u(x)(t)
\]
emphasizing its interpretation as a $t$-parametrized curve in some space of functions defined on $\Omega$. According to this interpretation we will sometimes write
\[
\dot u (x)(\cdot)= \frac{\partial u}{\partial t}(x, \cdot) .
\]
We are ready to state our main result:
\begin{theorem}\label{principale}
Let $f , a, \sigma , h \in C^5(\mathbb{R}^n , \mathbb{R})$,  $g \in C^4(\mathbb{R}^n , \mathbb{R})$,  $w \in C^5(\mathbb{R}^n \times (-T , T) , \mathbb{R})$, $w$ having compact support. and consider their restrictions as the corresponding functions appearing in $(L)$. Let $f^{(p)}, w^{(p)}, g^{(q)}$, $p=0,1, \ldots ,5 , q= 0,1, \dots, 4$ decay suitably fast as $\vert x \vert \rightarrow \infty$ in order that the Fourier transforms
\[
\hat{ f^{(p)} }( \alpha ) = \frac{1}{(2 \pi)^{\frac{n}{2}}} \int_{\mathbb{R}^n} f^{(p)}(s) e^{-i\alpha \cdot s} \; ds
\]
\[
\hat {g^{(q)}} ( \alpha ) = \frac{1}{(2 \pi)^{\frac{n}{2}}} \int_{\mathbb{R}^n} g^{(q)}(s) e^{-i\alpha \cdot s}  \; ds
\]
are well-defined: for instance let $f , w  \in C^5(\mathbb{R}^n , \mathbb{R}) \cap W^{5,1}(\mathbb{R}^n , \mathbb{R})$ and $g \in C^4(\mathbb{R}^n , \mathbb{R}) \cap W^{4,1}(\mathbb{R}^n , \mathbb{R})$. Notation $\hat f= \hat {f}^{(0)}$,  $\hat g= \hat {g}^{(0)}$ will be used.

 Let 
\[
v^{\Delta x , \Delta t} : \Sigma_{\Delta x , \Delta t} \cap (\Omega \times (-T , T)) \rightarrow \mathbb{R}
\]
be the solution of $(L)_{\Delta x , \Delta t}$,  $(\Delta x , \Delta t) \in \mathcal{A}_T$. Then
\begin{itemize}
\item [a.1)]
\begin{equation}\label{lewy}
\; lim_{(\Delta x , \Delta t) \rightarrow (0,0)}  v^{\Delta x , \Delta t}(x,t) = u (x,t)
\end{equation}
uniformly for $(x,t) \in E_{\infty}$.  \\
\item [a.2)]
The limit in (\ref{lewy}) can be extended as a uniform limit to $\Omega \times (-T.T)$. \\
\item [a.3)]
The previous statements hold for the difference quotients of $v^{\Delta x , \Delta t}$ entering in the definition of $\Box_{\Delta x, \Delta t}$, which converge uniformly in $\Omega \times (-T,T)$ to the partial derivatives of $u$ entering in the definition of $\Box$. Therefore $u(x,t)$ is a $C^2$-solution of $(L)$.

In the homogeneous case, with $a \equiv 1$, $\sigma \equiv 0$, we have  in $E_{\infty}$ the analytic expression
\[
v^{\Delta x , \Delta t} (x,t) = \frac{1}{(2 \pi)^{\frac{n}{2}}} \int_{\mathbb{R}^n} e^{i \alpha \cdot s} (\hat f (\alpha) \cos (\beta t) + \hat g(\alpha) \frac{\Delta t \sin (\beta t)}{\sin (\beta \Delta t) } )\; d \alpha
\]   
and in $\Omega \times (-T,T)$
\[
u(x,t) = \frac{1}{(2 \pi)^{\frac{n}{2}}} \int_{\mathbb{R}^n} e^{i \alpha \cdot s} (\hat f (\alpha) \cos (\alpha t) + \hat g(\alpha) \frac{\sin (\vert \alpha \vert t)}{\sin (\vert \alpha \vert) }) \; d \alpha
\]
where $(\alpha , \Delta x , \Delta t) \rightarrow \beta (\alpha , \Delta x , \Delta t)$ is a real analytic function in a neighbourhood of $\mathbb{R}^n \times \{ 0 \} \times \{ 0 \}$.
  \\

\item[b.1)] for  $(\Delta x , \Delta t) \in \mathcal{A}_T$ 
\[
 \varphi^{\Delta x}(x)(t) = lim_{\Delta t \rightarrow 0} v^{\Delta x , \Delta t}(x,t)  
\]
exists and is  uniform with respect to $(x,t) \in E_{\infty}$.  \\
\item[b.2)]
The limit in the previous statement  can be extended as a uniform limit to $\Omega \times (-T.T)$. \\
\item [b.3)]
The previous statements hold for the difference quotient $\delta_{\Delta t}^{-1} \circ \delta_{\Delta t} v^{\Delta x , \Delta t}$, which converges uniformly in $\Omega \times (-T,T)$ to the second derivative with respect to $t$ of  $ \varphi^{\Delta x}(x)(t)$ .

In the homogeneous case, with $a \equiv 1$, $\sigma \equiv 0$, in $E_{\infty}$
\[
\varphi^{\Delta x}(x)(t)  =  \frac{1}{(2 \pi)^{\frac{n}{2}}} \int_{\mathbb{R}^n} e^{i \alpha \cdot s}(\hat f (\alpha ) \cos (\beta (\alpha , \Delta x , 0) t) + \hat g (\alpha )\frac{\sin (\beta (\alpha , \Delta x , 0)t)}{\beta (\alpha , \Delta x , 0)}  ) \; d \alpha
\]
 \\
\item [b.4)]
\[
 lim_{\Delta x \rightarrow 0} \varphi^{\Delta x}(x)(t) = u(x,t)
\]
{\it i.e.}
\[
lim_{\Delta x \rightarrow 0} lim_{\Delta t \rightarrow 0 }   v^{\Delta x , \Delta t}(x,t) = u (x,t).
\] 
uniformly for $(x,t) \in \Omega \times (-T,T)$\\
\item [c)] 
for each fixed $x$ , $\varphi^{\Delta x}(x)(t) $ is solution of the Lagrange's ODE
\[
\begin{cases}
\ddot {\xi }(x) = a(x) \sum_{k=1}^n  \delta_{k , \Delta x}^{-1} \circ \delta_ {k , \Delta x} \xi (x) + \sigma (x) \xi (x) + w(x,t)  \\
\xi (x) (0) = f(x) \; x \in \Omega_{\Delta x}\\
\dot {\xi (x)}(0)=g(x) \; x \in \Omega_{\Delta x} \\
\xi (x)(t) \equiv 0 \; x \in \partial \Omega_{\Delta x} \times (-T , T)
\end{cases}
\]
\end{itemize}
\end{theorem}
\begin{remark} : statements $a.1) - a.3)$ are a slight generalization of results by Courant  {\it et al.} \cite{cfl} and  H. Lewy \cite{lewy}, see also \cite{john-adv} \textsection {7}:  in the quoted articles only the homogeneous case is considered and, more important,  the ratio $\frac{\Delta t}{\Delta x}$ is kept fixed in the convergence analysis, a condition which does not fit with statement $b.4)$. For this reason in statements $a.1)-a.3)$ we consider convergence analysis of $v^{\Delta x, \Delta t}$ to the solution $u(x,t)$ of the wave equation putting no restriction on the way $({\Delta x, \Delta t}) \rightarrow (0,0)$, except its appartenence to $\mathcal{A}_T$ . Statements $b.1) -b.4), c)$ are the sought generalization of Lagrange's 1759 Theorem. The hypotheses of the theorem are satisfied {\it e.g.} if all the functions have derivatives of any order and have compact support.
\end{remark}
The remaining part of this section is devoted to the proof of this theorem in the case of constant velocity ($a \equiv 1$), infinite flexibility ($\sigma \equiv 0$), absence of forcing term ($w \equiv 0$). We will mostly consider the case of a Cauchy problem, adding the necessary comments to deal with boundary conditions at the end of this section.

The proof begins with a lemma that, though not strictly necessary, makes clear the fundamental argument leading to Theorem (\ref{principale}). Here and throughout this section we will consider complex-valued functions $v: \mathbb{R}^n \rightarrow \mathbb{C}$, the possibility to get back to real-valued functions consisting in taking the real part of $v$.
\begin{lemma}\label{spettro-c}

\begin{itemize}
\item [a)]

 Among the couples $(eigenvalue/eigenvector)$ of  $\Box$ acting on complex-valued functions defined in $\mathbb{R}^n$ are
\[
(-\beta + \vert \alpha \vert^2 , e^{i(\alpha \cdot x + \beta t)})
\]
$x \in \mathbb{R}^n, \beta \in \mathbb{R}, \alpha \in \mathbb{R}^n$. \\
\item [b)] Among the solutions of $\Box u =0$ are the functions
\[
(x,t) \rightarrow e^{i(\alpha \cdot x ) - \vert \alpha \vert^2 t}
\] \\
\item[c)] The initial value problem $(E)$, with $\Omega = \mathbb{R}^n$ and $f,g$ satisfying the same hypotheses in Theorem (\ref{principale}), has for $(x,t) \in \Omega \times (-T,T)$ the unique solution
\begin{equation}\label{u}
u(x{},t) = \frac{1}{(2 \pi)^{\frac{n}{2}}} \int_{\mathbb{R}^n} e^{i \alpha \cdot x} (\hat f (\alpha) \cos (\vert \alpha \vert t) + \frac{\hat g (\alpha)}{\vert \alpha \vert} \sin (\vert \alpha \vert t)) \; d \alpha .
\end{equation}
\end{itemize}
\end{lemma}

\begin{proof}
Statements $a)$, $b)$ are a straightforward computation. Statement $c)$ follows observing that problem $(E)$ with initial data $f= e^{i \alpha \cdot x}, g \equiv 0$ has solution
\[
(x.t) \rightarrow e^{i \alpha \cdot x} \cos (\vert \alpha \vert t)
\]
while problem $(E)$ with data $f \equiv 0 , g= e^{i \alpha \cdot x}$ has solution
\[
e^{i \alpha \cdot x}  \frac{\sin (\vert \alpha \vert t)}{\vert \alpha \vert}
\]
Then by using linearity of the differential equations and regularity assumptions on data, which allows to exchange the order of application between the differential operator $\Box$ and the integral, statement $c)$ follows.
\end{proof}
The next lemma is analogous to the previous one: we just substitute the (continuous) d'Alembert operator with its discretized version $\Box_{\Delta x , \Delta t}$, and observe that eigenfunctions of such discretized d'Alembert operator are the same of the continuous version of it, while the eigenvalues depend analytically on the discretization parameters, hence allowing an explicit expression of the solutions of $(E)_{\Delta x , \Delta t}$ in term of the Fourier transform of the data. The precise statement is
\begin{lemma}\label{spettro-d}
Let $(\Delta x , \Delta t) \in \mathcal{A}_T$.
\begin{itemize}
\item [a)] Among the couples $(eigenvalue/eigenvector)$ of  $\Box_{\Delta x , \Delta t}$ acting on complex-valued functions defined in $\mathbb{R}^n$ are
\[
(G(\alpha , \beta^2 , \Delta x , \Delta t), e^{i(\alpha \cdot x + \beta t)})
\]
where
\[
G(\alpha , \beta^2 , \Delta x , \Delta t) = - \frac{\sin^2 (\frac{\beta \Delta t}{2})}{(\frac{\beta \Delta t}{2})^2} \beta^2 + \sum_{k=1}^n \frac{\sin^2 (\frac{\alpha_k \Delta x}{2})}{(\frac{\alpha_k \Delta x}{2})^2} \alpha_k^2
\] \\
\item [b.1)]Among the solutions of $\Box_{\Delta x , \Delta t} =0$ are the functions
\[
(x,t) \rightarrow e^{i(\alpha \cdot x + \beta t)}
\]
where $\beta^2 = \beta^2 (\alpha , \Delta x , \Delta t)$ and $ \beta (\alpha , \Delta x , \Delta t)$ is the positive real positive square root of the real analytic branch of solutions of $G(\alpha , \beta^2 , \Delta x , \Delta t) =0$. The fact that $\beta^2$ is real follows by (\ref{cfl})\\
\item [b.2)] The function $(\alpha . \Delta x . \Delta t) \rightarrow \beta (\alpha , \Delta x , \Delta t)$ defined in $b1)$ is real analytic in a neighbourhood of $\mathbb{R}^n \times \{ 0 \} \times \{ 0 \} \subset \mathbb{R}^n \times \mathbb{R} \times \mathbb{R}$, therefore for any fixed $M>0$ there exists $R>0$ such that $\beta (\alpha , \Delta x , \Delta t)$ is real analytic in $\{ \alpha \in \mathbb{R}^n : \vert \alpha \vert \leq M \} \times \{ \vert \Delta x \vert <R \}  \times \{ \vert \Delta x \vert <R \}$, and $\beta (\alpha , 0,0)= \vert \alpha \vert$. \\
\item [c)] The initial value problem $(E)_{\Delta x , \Delta t}$, with $f,g$ satisfying the hypotheses of Theorem (\ref{principale}) and with $(\Delta x , \Delta t) \in \mathcal{A}_T$, has the unique solution
\[
v^{\Delta x , \Delta t} (x,t) = \frac{1}{(2 \pi)^{\frac{n}{2}}} \int_{\mathbb{R}^n} e^{i \alpha \cdot x} (\hat f (\alpha) \cos (\beta t) + \hat g (\alpha) \frac{\Delta t}{\sin (\beta \Delta t) } \sin (\beta t) ) \; d \alpha
\]
where  $\beta = \beta (\alpha , \Delta x , \Delta t)$ and $(x,t) \in \Sigma_{\Delta x , \Delta t} \cap (\Omega \times (-T,T))$.
\end{itemize}
\end{lemma}
\begin{proof}
A straightforward computation gives
\[
\delta_{\Delta t}^{-1} \circ \delta_{\Delta t} e^{i (\alpha \cdot x + \beta t)} = - e^{i (\alpha \cdot x + \beta t)}  \frac{\sin^2 \frac{\beta \Delta t}{2}}{(\frac{\beta \Delta t}{2})^2} \beta^2
\]
and analogously
\[
\delta_{k , \Delta x , }^{-1} \circ \delta_{k , \Delta x} e^{i (\alpha \cdot x + \beta t)} = - e^{i (\alpha \cdot x + \beta t)}  \frac{\sin^2 \frac{\alpha_k \Delta x}{2}}{(\frac{\alpha_k \Delta x}{2})^2} \alpha_k^2
\]
and statements $a)$, $b.1)$ follow.

From (\ref{cfl}) and
\[
\begin{cases}
G(\alpha , \vert \alpha \vert^2 , 0,0)=0 \\
\frac{\partial G}{\partial (\beta^2)}(\alpha , \vert \alpha \vert^2, 0,0)= -1
\end{cases}
\]
the equation $G(\alpha , \beta^2 , \Delta x , \Delta t) =0$ defines implicitly  a {\it real} analytic  branch of $\beta^2 = \beta^2 (\alpha . \Delta x , \Delta t)$  emanating from $\vert \alpha \vert^2$ whose positive square root is $\beta (\alpha , \Delta x , \Delta t)$:  the fact that the equation $G=0$ has real solutions follows from Courant-Friedrichs-Lewy condition (\ref{cfl}). The rest of statement $b.2)$ follows from elementary geometric properties of the analytic set defined by $G=0$. Finally the proof of statement $c)$ is analogous to statement $c)$ of the previous lemma. Perhaps the only useful remark is that, knowing that
\[
(x,t) \rightarrow c e^{i \alpha \cdot x} \sin (\beta t)
\]
is a solution of $\Box_{\Delta x , \Delta t} v =0$ for any real constant $c$, satisfying $v(x.0)=0$, the condition $\delta_{\overline {\Delta t}} v (x,0) = e^{i \alpha \cdot x}$ reads as
\[
c  e^{i \alpha \cdot x} \frac{\sin (\beta \Delta t) - \sin ( -\beta \Delta t)}{2 \Delta t} =  e^{i \alpha \cdot x}
\]
which determines $c= \frac{\Delta t}{\sin (\beta \Delta t)}$ and proves $c)$: the validity of the analytic form of $v^{\Delta x, \Delta t}$ in the domain $\Sigma_{\Delta x , \Delta t} \cap \Omega \times (-T,T)$ follows from (\ref{cfl}).
\end{proof}
To prove statement $a.1)$ of Theorem (\ref{principale}) for the simplified inital data problem $(E)$ and its approximation $(E)_{\Delta x \Delta t}$ we must show that for any fixed $T>0$
\[
\forall \varepsilon >0 \; \exists \; \overline {\Delta x}=\overline {\Delta x}(\varepsilon , T), \overline {\Delta t}=\overline {\Delta t}(\varepsilon , T)>0
\]
such that if $\frac{\Delta t}{\Delta x} \leq \frac{1}{\sqrt{n}}$, $0<\Delta t <\overline {\Delta t}$, $0<\Delta x <\overline {\Delta x}$, $\frac{T}{\Delta t} \in \mathbb{N}$ then
\[
\vert v^{\Delta x , \Delta t}(x.t) - u(x,t) \vert < \varepsilon
\]
uniformly for $(x,t) \in E_{\infty}$. Using the analytic expression of $u(x,t)$ we get
\[
\vert  v^{\Delta x , \Delta t}(x.t) - u(x,t) \vert \leq \frac{1}{(2 \pi)^{\frac{n}{2}}} \int_{\mathbb{R}^n} (\vert \hat f (\alpha) \vert \vert \cos (\beta t) - \cos (\vert \alpha \vert t) \vert + \vert \hat g (\alpha) \vert \vert \frac{\Delta t \sin (\beta t)}{\sin (\beta \Delta t)}  - \frac{\sin (\vert \alpha \vert t)}{\vert \alpha \vert}\vert) \; d \alpha =
\]
\[
\frac{1}{(2 \pi)^{\frac{n}{2}}} \int_{\vert \alpha \vert \leq M} \Theta \; d \alpha + \frac{1}{(2 \pi)^{\frac{n}{2}}} \int_{\vert \alpha \vert > M} \Theta \; d \alpha
\]
where $M>0$. We observe that $t \in \Sigma_{\Delta t} \cap (-T,T)$ and  $T=N \Delta t$, $N \in \mathbb{N}$, therefore $t=k \Delta t$, $k \in \mathbb{Z}, \vert k \vert \leq N$. Then
\[
\sin (\beta t) = \sin (\beta k \Delta t) = \sin (\beta (k-1) \Delta t + \beta \Delta t)
\]
therefore
\[
\vert \frac{\sin (\beta t)}{\sin (\beta \Delta t)} \vert \leq 1 + \vert \frac{\sin (\beta (k-1) \Delta t)}{\sin (\beta \Delta t)} \vert
\]
whose iteration $k$-times gives
\begin{equation}\label{seno}
\vert \frac{\Delta t \sin (\beta t)}{\sin (\beta \Delta t)} \vert \leq T
\end{equation}
therefore
\[
\Theta \leq 2\vert \hat f (\alpha) \vert + \vert \hat g (\alpha) \vert T .
\]
The hypotheses on $f,g$ imply that $\vert \hat f (\alpha) \vert, \vert \hat g (\alpha) \vert$ decays sufficiently fast in order that there exists $M=M(\varepsilon  ,T )$ such that
\[
\frac{1}{(2 \pi)^{\frac{n}{2}}} \int_{\vert \alpha \vert > M} \Theta \; d \alpha < \frac{\varepsilon}{2}.
\]
For such fixed $M$ we must find sufficiently small $\overline {\Delta x} $, $\overline {\Delta t} $ such that
\[
\frac{1}{(2 \pi)^{\frac{n}{2}}} \int_{\vert \alpha \vert \leq M} \Theta \; d \alpha  < \frac{\varepsilon}{2}
\]
if $0<\Delta x < \overline {\Delta x}$, $0<\Delta t < \overline {\Delta t}$,  and (\ref{cfl}) holds. The Fourier transforms of data are bounded on $\mathbb{R}^n$, and from Lemma (\ref{spettro-d}) the function $(\alpha , \Delta x , \Delta t) \rightarrow \beta (\alpha , \Delta x , \Delta t)$ is analytic in a neighbourhood of $\{ \vert \alpha \vert \leq M\} \times [- \overline {\Delta x} , \overline {\Delta x}] \times [- \overline {\Delta t} , \overline {\Delta t}]$, and the last inequality follows from uniform continuity in $\mathbb{R}$ of trigonometric functions. The proof of statement $a.1)$ in Theorem (\ref{principale}) is concluded: for the considered simplified version of $(L)$, and with fixed ratio $\frac{\Delta t}{\Delta x}$, it  is due to Lewy \cite{lewy}, see also \cite{john-adv} \textsection $7.3$. We remark that convergence is uniform for $(x,t) \in E_{\infty}$. The last argument, based on the splitting of the Fourier transform expression of the difference between an approximated and a limit solution in a high frequency part, uniformly estimated for given $T$ by the fast decay of Fourier trasform of data, and in a bounded frequency part, estimated by uniform continuity of trigonometric function and analytic extension up to $0$-value of the discrtization parameters of the frequency function of the approximated solution, will be used several times in the proof: we will refer to it as {\it frequency splitting argument}.

Let $(x,t) 	\in \Omega \times (-T,T) - E_{\infty}$: there exists a sequence $(x, t_p) \in \Sigma_{\Delta x , {\Delta t}_p }\in \mathcal{A}_T$ such that $(x, t_p) \rightarrow (x,t)$ and ${\Delta t}_p =\frac{ \Delta t}{2^p}$ for a given $\Delta t$. Let $v^p = v^{\Delta x , {\Delta t}_p}$. Then
\[
\vert v^p (x,t_p) -u(x,t) \vert \leq \vert v^p (x,t_p) - u(x,t_p) \vert + \vert u(x,t_p) - u(x,t) \vert 
\]
The first term of the sum in the {\it r.h.s.} of the last inequality can be made as small as we wish independently of $(x,t_p) \in E_{\infty}$ just choosing $p$ sufficiently big, as proved before. The second term in the {\it r.h.s.} of the last equality, which is defined in $\Omega \times (-T,T)$, can be made as small as we wish using the frequency splitting argument and uniform continuity of the integrand in the analytic expression of $u(.,.)$: this ends the proof of statement $a.2)$. Incidentally, the analytic expression of the limit $u(.,.)$ proves the this extension of the limit of the solution of $(E)_{\Delta x, \Delta t}$ from $E_{\infty}$ to $\Omega \times (-T,T)$ is unique.  This argument does not request all the regularity assumptions in Theorem (\ref{principale}): the higher regularity hypotheses are used when the same argument is applied to the difference quotients of $v^{\Delta x , \Delta t}$ to get the conclusions in statement $a.3)$.
To prove statements $b.1) - b.3)$ of Theorem (\ref{principale}) we will use the analytic expression obtained in Lemma (\ref{spettro-d}) and prove that, remembering that when $(\Delta x , \Delta t) \in \mathcal{A}_T$ then $\Delta t = \frac{T}{N}$, $N \in \mathbb{N}$
\begin{equation}\label{zero}
\begin{split}
&lim_{\Delta t \rightarrow 0} v^{\Delta x , \Delta t} =lim_{\Delta t \rightarrow 0}  \frac{1}{(2 \pi)^{\frac{n}{2}}} \int_{\mathbb{R}^n} e^{i \alpha \cdot x}(\hat f(\alpha) \cos (\beta t) + \hat g (\alpha) \frac{\Delta t}{\sin (\beta \Delta t)} \sin (\beta t)) \; d \alpha =
\\
&\frac{1}{(2 \pi)^{\frac{n}{2}}} \int_{\mathbb{R}^n} e^{i \alpha \cdot x}(\hat f(\alpha) \cos (\beta_0 t) + \hat g (\alpha) \frac{1}{\beta_0 } \sin (\beta_0 t)) \; d \alpha = \varphi^{\Delta x}(x)(t)
\end{split}
\end{equation}
where the last equality is the definition of $\varphi^{\Delta x}(x)(t)$ and
\[
\beta_0 = \beta (\alpha , \Delta x , 0)= \sqrt{\sum_{k=1}^n \frac{\sin^2 (\frac{\alpha_k \Delta x}{2})}{(\frac{\alpha_k \Delta x}{2})^2} \alpha_k^2} .
\]
To prove the above equalities we must prove that
\[
\forall \varepsilon >0 \; \exists \overline {\Delta t}=\overline {\Delta t}(\varepsilon , T)>0
\]
such that $\forall \Delta t \in ]0 , \overline {\Delta t}[$ we have
\[
\vert \int_{\mathbb{R}^n} e^{i \alpha \cdot x}(\hat f(\alpha) (\cos (\beta t)  - \cos (\beta_0 t))+ \hat g (\alpha) (\frac{\Delta t}{\sin (\beta \Delta t)} \sin (\beta t) - \frac{1}{\beta_0 } \sin (\beta_0 t)))  \; d \alpha  \vert < \varepsilon
\]
As we did in proving statement $a.1)$ we use the frequency splitting argument observing that
\[
\vert \int_{\mathbb{R}^n} e^{i \alpha \cdot x}(\hat f(\alpha) (\cos (\beta t)  - \cos (\beta_0 t))+ \hat g (\alpha) (\frac{\Delta t}{\sin (\beta \Delta t)} \sin (\beta t) - \frac{1}{\beta_0 } \sin (\beta_0 t)))  \; d \alpha  \vert <
\]
\[
 \int_{\mathbb{R}^n} (\vert \hat f(\alpha) \vert \vert\cos (\beta t)  - \cos (\beta_0 t)\vert+ \vert \hat g (\alpha) \vert \vert \frac{\Delta t}{\sin (\beta \Delta t)} \sin (\beta t) - \frac{1}{\beta_0 } \sin (\beta_0 t))\vert  \; d \alpha  =
\]
\[
= \int_{\mathbb{R}^n} \Lambda \; d \alpha = \int_{\vert \alpha \vert \leq M} \Lambda \; d \alpha + \int_{\vert \alpha \vert >M} \Lambda \; d \alpha 
\]
for $M>0$.  Using that for any $z \in \mathbb{R}$ one has $\vert \frac{\sin z}{z} \vert \leq 1$, choosing $t$ such that $\vert t \vert \leq T$ and using (\ref{seno}) we get
\[
\vert   \frac{\Delta t}{\sin (\beta \Delta t)} \sin (\beta t) - \frac{1}{\beta_0 } \sin (\beta_0 t))  \vert \leq 2T
\]
hence there exists $M=M(\varepsilon , T)$ such that
\[
\int_{\vert \alpha \vert >M} \Lambda \; d \alpha  < \frac{\varepsilon}{2}.
\]
Applying the same argument used before, we conclude from analytic extension up to $\Delta x =0$, $\Delta t =0$ of $\beta (\alpha , \Delta x , \Delta t)$, that there exist a positive constant $\overline {\Delta t}$ such that if $\Delta t < \overline {\Delta t}$ then
\[
\int_{\vert \alpha \vert \leq M} \Lambda \; d \alpha  < \frac{\varepsilon}{2}
\]
and $b.1)$ has been proved.

$b.2) $ is proved in the same way we proved $a.2)$: here is crucial to observe that writing
\[
\varphi^{\Delta x}(x)(t)  =  \frac{1}{(2 \pi)^{\frac{n}{2}}} \int_{\mathbb{R}^n} e^{i \alpha \cdot s}(\hat f (\alpha ) \cos (\beta (\alpha , \Delta x , 0) t) + \hat g (\alpha )\frac{\sin (\beta (\alpha , \Delta x , 0)t)}{\beta (\alpha , \Delta x , 0) t}  )  t\; d \alpha
\]
such analytic expression is defined in $\Omega \times (-T.T)$.

The proof of $b.3)$ is then similar to that of $a.3)$

To prove $b.4)$ we need a lemma, analogous to  Lemma (\ref{spettro-c}) and  Lemma (\ref{spettro-d}),  for the mixed continuous/discrete differential operator
\[
\Diamond  = \frac{d^2}{d t^2} -  \sum_{k=1}^n  \delta_{k , \Delta x}^{-1} \circ \delta_ {k , \Delta x}
\]
\begin{lemma}\label{spettro-cd}
\begin{itemize}
\item [a)] Among the couples $(eigenvalue/eigenvector)$ of $\Diamond$ there are
\[
(-\beta^2 +  \sum_{k=1}^n \frac{\sin^2 (\frac{\alpha_k \Delta x}{2})}{(\frac{\alpha_k \Delta x}{2})^2} \alpha_k^2) , e^{i(\alpha \cdot x + \beta t)}
\]
where from (\ref{cfl}) $\beta \in \mathbb{R}^+$ and $\alpha \in \mathbb{R}^n$. \\
\item[b)] Among the solutions of $\Diamond \xi =0$ are the functions
\[
(x,t) \rightarrow e^{i(\alpha \cdot x + \beta_0 t)}
\]
where $\beta_0 = \sqrt{\frac{\sin^2 (\frac{\alpha_k \Delta x}{2})}{(\frac{\alpha_k \Delta x}{2})^2} \alpha_k^2) }$ \\
\item [c)] The initial value problem (Lagrange's model for the homogeneous, constant velocity, infinite flexibility, the wave equation)
\begin{equation}\label{diamante}
\begin{cases}
\Diamond \xi (x)=0  \; in \; \mathbb{R}^n \times \mathbb{R} \\
\xi (x)(0)= f(x) \; in \; \mathbb{R}^n \times \{ 0 \} \\
\dot \xi (x) (0) = g (x) \; in \; \mathbb{R}^n \times \{ 0 \}
\end{cases}
\end{equation}
with $f, g$ as in Theorem (\ref{principale}), has solution
\begin{equation}\label{fi}
\varphi^{\Delta x}(x)(t) = \frac{1}{(2 \pi)^{\frac{n}{2}}} \int_{\mathbb{R}^n} e^{i \alpha \cdot x} (\hat f (\alpha) \cos (\beta_0 t) + \frac{\hat g (\alpha)}{\beta_0} \sin (\beta_0 t)) \; d \alpha 
\end{equation}
which is defined in $\Omega \times (-T,T)$.
\end{itemize}
\end{lemma}
\begin{proof}
sSatements $a)$, $b)$ are straightforward computations. Statement $c)$ follows computing solutions of (\ref{diamante}) with data $f(x)= e^{i \alpha \cdot x}$, $g \equiv 0$, respectively $f \equiv 0$, $g(x)= e^{i \alpha \cdot x}$: then using linearity of the equation and regularity assumptions on data, which allow the passage of derivatives up to second order of the function defined by (\ref{fi}) inside the integral in its definition, ends the proof of statement $c)$.
\end{proof}
Statement $cb.4$ of Theorem (\ref{principale}) for the homogeneous initial data problem , with $a \equiv 1$, $\sigma \equiv 0$, then follows from (\ref{zero}), (\ref{fi}). To complete the proof of Theorem (\ref{principale}) in this setting we must prove that
\[
\forall \varepsilon>0 \; \exists \overline {\Delta x} = \overline {\Delta x}(\varepsilon , T) >0
\]
such that if $\Delta x <  \overline {\Delta x}$ then
\begin{equation}\label{ultima}
\vert \varphi^{\Delta x}(x)(t) - u(x,t) \vert < \varepsilon
\end{equation}
where $u(x,t)$ is the solution of $(E)$.  From (\ref{u}) and (\ref{fi})
\[
\vert \varphi^{\Delta x}(x)(t) - u(x,t) \vert \leq \frac{1}{(2 \pi)^{\frac{n}{2}}} \int_{R}^n (\vert \hat f (\alpha) \vert   \vert \cos (\beta_0 t) - \cos (\vert \alpha \vert t) \vert + \vert \hat g (\alpha) \vert \vert \frac{\sin (\beta_0 t)}{\beta_0} - \frac{\sin (\vert \alpha \vert t)}{\vert \alpha \vert} \vert) \; d \alpha=
\]
\[
 \frac{1}{(2 \pi)^{\frac{n}{2}}} \int_{R}^n \Gamma \; d \alpha = \frac{1}{(2 \pi)^{\frac{n}{2}}} \int_{\{ \vert \alpha \vert \leq M \}}\Gamma \; d \alpha +\frac{1}{(2 \pi)^{\frac{n}{2}}} \int_{\{ \vert \alpha \vert >  M \}}\Gamma \; d \alpha
\]
Once again we use the frequency splitting argument: from (\ref{seno}), $\vert \frac{\sin z}{z} \vert \leq 1$, $\vert t \vert \leq T$ we get
\[
\Gamma \leq 2 \vert \hat f (\alpha) \vert + 2 \vert \hat g (\alpha) \vert T
\]
therefore there exists $M=M(\varepsilon , T)>0$ such that $\frac{1}{(2 \pi)^{\frac{n}{2}}} \int_{\{ \vert \alpha \vert >  M \}}\Gamma \; d \alpha < \frac{\varepsilon}{2}$. Moreover boundness of Fourier transforms of data, analyticity up to $\Delta x = \Delta t =0$ of $\beta (\alpha , \Delta x ,\Delta t)$ and
\[
lim_{\Delta x \rightarrow 0} \beta_0 (\alpha , \Delta x , 0)= \vert \alpha \vert
\]
uniformly for $\vert \alpha \vert \leq M$, implies the existence of $\overline {\Delta x} = \overline {\Delta x} (\varepsilon ,T) >0$ such that if $\Delta x < \overline {\Delta x}$ (\ref{ultima}) holds. This ends the proof of Theorem (\ref{principale}) for a homogeneous, simplified, initial value problem. 

The case of an equally simplified homogeneous mixed initial/boundary problem is dealt with in the same way we treated the initial data problem:  the boundary condition, when translated into an assignment of the values of the solutions to the mixed problem analogous to $(E)_{\Delta x , \Delta t}$,  leads to the same analytic expressions of $v^{\Delta x , \Delta t}(x,t)$, $u(x,t)$, $\varphi^{\Delta x}(x)(t)$ we found for the case of the purely initial value problem. One has only to avoid possible ambiguities in the definition of points in $\Omega_{\Delta x}$ and in $\partial_{\Delta x} \Omega$ which could occur, for instance from the existence of "double points" in the boundary $\partial \Omega$, see \cite{cfl}. Here a double point $x \in \partial \Omega$ is a point such that for any ball $B$ centered at $x$ the set $B \cap \Omega$ is not connected. In any case, the conclusion for the homogeneous, constant velocity, infinitely flexible case with mixed initial/boundary condition are, for the case of smooth boundary $\partial \Omega$, those of Theorem (\ref{principale}): they can be easily extended to the case of corners in the boundary of $\Omega$ when the boundary has no double points.
\begin{remark} The interpretation of the solution $u(x,t)$ of, say, $(E)$ as the iterated limit
\[
lim_{\Delta x \rightarrow 0} lim_{\Delta t \rightarrow 0}v^{\Delta x , \Delta t}(x,t) = u(x,t)
\]
could suggest an analogous property obtained by inversion of the order of limits, whose possible interpretation is the convergence of the Euler broken line approximation of the evolution ODE equivalent to the wave equation: this claim is actually false, in general, for the Courant-Friedrichs-Lewy condition (\ref{cfl}) clearly shows that the role played by the spatial and time discrtization parameters is not symmetric.
\end{remark}
\section{The inhomogeneous case}
In this section we will prove Theorem (\ref{principale}) in the same simplified version considered in the previous one, except for the substitution in $(E)$ of the homogeneous differential equation with the inhomogeneous one
\[
\Box u = w \; in \; \Omega \times (-T,T) .
\]
As in the previous section we will mainly pay attention to the initial data case. We will reduce the inhomogeneous Cauchy problem  to a family of homogeneous ones {\it via} variation of consatnt method (Duhamel principle), hence rededucing its  proof  to that of the homogeneous case.

Firstly we write the inhomogeneous form of $(E)$ as an equivalent first order system of PDE
\[
(EV)
\begin{cases}
\dot \xi (x) = A \xi (x) + W(x) \; in \; \Omega \times (-T,T) \\
\xi_1(x)(0)= f(x) \; in \; \Omega \\
\dot \xi_2 (x)(0)=g(x) \; in \; \Omega 
\end{cases}
\]
where $\xi ={}^t(\xi_1 , \xi_2)$, $A \xi (x)= {}^t(\xi_2 , \Delta \xi_1)$ and $W(x)(\cdot) = {}^t(0 , w(x)(\cdot))$
The discretized version of this problem, which is equivalent to $(E)_{\Delta x , \Delta t}$, is
\[
(EV)_{\Delta x , \Delta t}
\begin{cases}
\delta_{\Delta t} \xi (x)(t) = A_{\Delta x} \xi (x) (t)+ W(x)(t) \; in \; \Omega \times (-T,T) \\
\xi_1(x)(0)= f(x) \; in \; \Omega \\
\delta_{\overline {\Delta t}} \xi_2 (x)(0)=g(x) \; in \; \Omega 
\end{cases}
\]
where $A_{\Delta x}$ is obtained from $A$ substituting the Laplace operator with respect to the spatial variable with its discretized version $\sum_{k=1}^n  \delta_{k , \Delta x}^{-1} \circ \delta_ {k , \Delta x}$ and all the functions appearing in $(EV)_{\Delta x , \Delta t}$ are evaluated in the lattices $\Sigma_{\Delta x , \Delta t}$, $\Sigma_{\Delta x}$, $(\Delta x , \Delta t) \in \mathcal{A}_T$.

Let $\eta^{\Delta x , \Delta t}(x)(t)$ be the solution of homogeneous case ($W \equiv 0$) of $(EV)_{\Delta x , \Delta t}$: from the theory developed in the previous section it has the form
\[
\eta^{\Delta x , \Delta t}(x)(t) = \int_{\mathbb{R}^n} \mathcal{W}^{\Delta x , \Delta t}(x, \alpha ) (t) 
\begin{pmatrix}
\hat f (\alpha) \\
\hat g (\alpha)
\end{pmatrix}
\; d \alpha
\]
where
\[
\mathcal{W}^{\Delta x , \Delta t}(x, \alpha ) (t)  = \frac{e^{i \alpha \cdot x}}{(2 \pi)^{\frac{n}{2}}}
\begin{pmatrix}
\cos (\beta t)  & \frac{\Delta t \sin (\beta t) }{\sin (\beta \Delta t)} \\
- \frac{\sin (\beta  t)}{\beta} & \cos (\beta t) \frac{\beta \Delta t}{\sin (\beta \Delta t)}
\end{pmatrix}
\]
$\beta = \beta (\alpha , \Delta x , \Delta t)$ being defined in Lemma (\ref{spettro-d}). Analogously the solution of $(EV)$ in the homogeneous case is
\[
\eta (x)(t) = \int_{\mathbb{R}^n} \mathcal{W}(x, \alpha ) (t) 
\begin{pmatrix}
\hat f (\alpha) \\
\hat g (\alpha)
\end{pmatrix}
\; d \alpha
\]
where
\[
\mathcal{W} (x, \alpha ) (t)  = \frac{e^{i \alpha \cdot x}}{(2 \pi)^{\frac{n}{2}}}
\begin{pmatrix}
\cos (\vert \alpha \vert t)  & \frac{ \sin (\vert \alpha t \vert) }{\vert \alpha \vert} \\
- \vert \alpha \vert \sin (\vert \alpha \vert t) & \cos (\vert \alpha \vert t) .
\end{pmatrix}
\]
With these notations statement $a.1)$ of Theorem (\ref{principale}), proved in the present homogeneous case in the previous section, implies
\[
lim_{\Delta x , \Delta t \rightarrow (0,0)}  \mathcal{W}^{\Delta x , \Delta t}(x, \alpha ) (t) = \mathcal{W}(x, \alpha ) (t) 
\]
uniformly for $(x,t) \in E_{\infty}$.

We define the mixed discrete/continuous Cauchy problem
\[
(EV)_{\Delta x}
\begin{cases}
\dot \xi (x)(t) = A_{\Delta x} \xi (x) (t)+ W(x)(t) \; in \; \Omega \times (-T,T) \\
\xi_1(x)(0)= f(x) \; in \; \Omega \\
\dot \xi_2 (x)(0)=g(x) \; in \; \Omega 
\end{cases}
\]
whose solution is
\[
\varphi^{\Delta x} (x)(t) = \int_{\mathbb{R}^n} \mathcal{W}^{\Delta x}(x, \alpha ) (t) 
\begin{pmatrix}
\hat f (\alpha) \\
\hat g (\alpha)
\end{pmatrix}
\; d \alpha
\]
In the last section we proved that
\[
\begin{cases}
lim_{\Delta t \rightarrow 0} \mathcal{W}^{\Delta x , \Delta t}(x,\alpha)(t) =\mathcal{W}^{\Delta x} (x, \alpha)(t) \\
lim_{\Delta x \rightarrow 0} \mathcal{W}^{\Delta x }(x,\alpha)(t) =\mathcal{W} (x, \alpha)(t)
\end{cases}
\]
uniformly for $(x,t) \in E_{\infty}$.
By linearity of the wave equation, to prove Theorem (\ref{principale}) in the inhomogeneous case it is sufficient to prove it when the initial data are $f \equiv 0$, $g \equiv 0$. By the classical theory of variation of constants applied to  $(E)_{\Delta x , \Delta t}$ we get that the solution of $(E)_{\Delta x , \Delta t}$ with such initial data is
\[
\begin{cases}
(x,t) \rightarrow \int_0^t \mathcal{W}^{\Delta x , \Delta t}(x)(t-s) W(x)(s) \; d s \\
\mathcal{W}^{\Delta x , \Delta t}(x)(t-s) = \int_{\mathbb{R}^n} \mathcal{W}^{\Delta x , \Delta t}(x, \alpha)(t-s) W(x, \alpha)(s)  \; d \alpha
\end{cases}
\]
where {\it e.g.} $W(x, \alpha)(s)$ is the Fourier transform of $W(x)(s)$ with respect to the $x$-variable.
Analogously the solution of  $(E)_{\Delta x}$ with null initial data is
\[
\begin{cases}
(x,t) \rightarrow \int_0^t \mathcal{W}^{\Delta x }(x)(t-s) W(x)(s) \; d s \\
\mathcal{W}^{\Delta x }(x)(t-s) = \int_{\mathbb{R}^n} \mathcal{W}^{\Delta x}(x, \alpha)(t-s) W(x)(s).
\end{cases}
\]
Theorem (\ref{principale}) impies that if $(\Delta x,\Delta t) \in \mathcal{A}_T$, uniformly for $(x,t) \in E_{\infty}$
\[
lim_{(\Delta x,\Delta t) \rightarrow (0,0)}\int_0^t \mathcal{W}^{\Delta x , \Delta t}(x)(t-s) W(x)(s) \; d s  = \int_0^t \mathcal{W}(x)(t-s) W(x)(s) \; d s 
\]
\[
lim_{\Delta t \rightarrow 0 } \int_0^t \mathcal{W}^{\Delta x , \Delta t}(x)(t-s) W(x)(s) \; d s =\int_0^t \mathcal{W}^{\Delta x}(x)(t-s) W(x)(s) \; d s 
\]

\[
lim_{\Delta x \rightarrow 0} \int_0^t \mathcal{W}^{\Delta x }(x)(t-s) W(x)(s) \; d s = \int_0^t \mathcal{W}(x)(t-s) W(x)(s) \; d s .
\]
On the other hand, the regularity assumptions on data and forcing term imply that
\[
\frac{d}{dt} \int_0^t \mathcal{W}(x)(t-s) W(x)(s) \; d s  = W(x)(t) + \int_0^t \frac{d}{dt} \mathcal{W}(x)(t-s) W(x)(s) \; d s 
\]
and the definition of $\mathcal{W}(x)(t)$ implies that
\[
\begin{cases}
\frac{d}{dt}\mathcal{W}(x)(t) c = A \mathcal{W}(x)(t) c \\
\mathcal{W}(x)(0) c = c
\end{cases}
\]
therefore

\[
\frac{d}{dt}(\mathcal{W}(x)(t)) c + \int_0^t
\mathcal{W}(x)(t-s)W(x)(s) \; ds ) = 
A \mathcal{W}(x)(t) c + W(x)(t) +\int_0^t A \mathcal{W}(x)(t-s) W(x)(s) \; d s =
\]
\[
 A (\mathcal{W}(x)(t) c + \int_0^t \mathcal{W}(x)(t-s) W(x)(s) \; d s ) + W(x)(t),
\]
{\it i.e.} the function
\[
t \rightarrow \mathcal{W}(x)(t)) c + \int_0^t
\mathcal{W}(x)(t-s)W(x)(s) \; ds
\]
is solution of $(EV)$. The special case when $c=0$ proves Theorem (\ref{principale}) for initial data $f \equiv 0 , g \equiv 0$, and therefore it ends the proof of such theorem  in the inhomogeneous, constant velocity, infinitely flexible, initial data case. Finally applying the same argument explained at the end of the last section extends the conclusion to the mixed initial/boundary case.

\section{Conclusion of the proof of the main theorem}

To finish the proof of Theorem (\ref{principale}) we write in the equation in $(L)$ the velocity function as
\[
a(x) = 1 + b(x)
\]
and write the solution $u(x,t)$ of $(L)$ as
\[
u(x,t)= \phi (x,t) + v(x)
\]
where $v(\cdot )$ is solution of
\[
(EL)
\begin{cases}
b(x) \Delta v = \sigma (x) v \; in \; \Omega \\
v(x)= h(x) \; in \; \partial \Omega
\end{cases}
\]
and $\phi (\cdot , \cdot )$ is solution of
\[
(L)'
\begin{cases}
\Box \phi = w \; in \; \Omega \times (-T , T) \\
\phi (x,0) = f(x) - v(x) \; in \; \Omega \times \{ 0 \} \\
\frac{\partial \phi}{\partial t}(x,0)= g(x) \; in \; \Omega \times \{ 0 \} \\
\phi (x,t) = 0 \; in \; \partial \Omega \times (-T,T) .
\end{cases}
\]
The compatibility conditions for $(L)'$ follow from those of $(L)$. The problems $(EL)$, $(L)'$ have discrtized versions
\[
(EL)_{\Delta x }
\begin{cases}
b(x) \sum_{k=1}^n  \delta_{k , \Delta x}^{-1} \circ \delta_ {k , \Delta x} v = \sigma (x)  \; in \; \Sigma_{\Delta x} \cap \Omega \\
v(x)= h(x) \; in \; \partial \Omega_{\Delta x}
\end{cases}
\]
and 
\[
(L)'_{\Delta x , \Delta t}
\begin{cases}
\Box \phi = w \; in \; \Sigma_{\Delta x , \Delta t} \cap \Omega \times (-T , T) \\
\phi (x,0) = f(x) - v(x) \; in \; \Sigma_{\Delta x} \cap \Omega \times \{ 0 \} \\
\frac{\partial \phi}{\partial t}(x,0)= g(x) \; in \; \Sigma_{\Delta x} \cap \Omega \times \{ 0 \} \\
\phi (x,t) = 0 \; in \; \partial \Omega-{\Delta x} \times (\Sigma_{\Delta t}(-T,T)) .
\end{cases}
\]
The existence of the solution $v^{\Delta x} : \Sigma_{\Delta x \cap \Omega \rightarrow \mathbb{R}}$ of $(EL)_{\Delta x}$ was originally  proved in \cite{cfl} and it is a standard result in numerical theory of elliptic PDEs. The existence of the solution of $(L)'_{\Delta x , \Delta t}$ follows by the explicit nature of the numerical scheme, and from (\cite{cfl}), and
\[
lim_{(\Delta x , \Delta t)\rightarrow (0,0)} \phi^{(\Delta x , \Delta t)}(x,t) = \phi (x,t)
\] 
is proved in the previous sections.

The Lagrange's discrete mechanical model equation
\[
(L)'_{\Delta x}
\begin{cases}
\ddot \xi (x) = \sum_{k=1}^n  \delta_{k , \Delta x}^{-1} \circ \delta_ {k , \Delta x} \xi (x) + w(x) \; in \; \Sigma _{\Delta x , \Delta t} \cap (\Omega \times (-T,T)) \\
\xi (x) (0) = f(x) - v^{\Delta x}(x) \; in \; \Sigma_{\Delta x}\cap \Omega \\
\dot \xi (x)(0) = g(x) \; in \; \Sigma_{\Delta x}\cap \Omega \\
\xi (x)(t) \equiv 0 \; in \; \partial \Omega_{\Delta x} \times (-T,T)
\end{cases}
\]
has solution $\varphi^{\Delta x}(x)(t)$ which, as proved in the previous sections, satisfies
\[
\begin{cases}
lim_{\Delta t \rightarrow 0} \phi^{\Delta x , \Delta t}(x,t) = \varphi^{\Delta x}(x)(t) \\
lim_{\Delta x \rightarrow 0} \varphi^{\Delta x}(x)(t)= \phi (x,t)
\end{cases}
\]
uniformly for $(x,t) \in \Omega \times (-T,T)$, hence
\[
lim_{\Delta x \rightarrow 0} \varphi^{\Delta x}(x)(t) + v^{\Delta x}(x) = u(x,t)
\]
uniformly for $(x,t) \in \Omega \times (-T,T)$. The proof of Theorem (\ref{principale}) is concluded.

As a final comment we observe that, as mentioned at the end of the second section, Theorem (\ref{principale}) is valid even if $\partial \Omega$ has corners ( {\it i.e.} points  where the tangent space to the boundary does not exists, but it does exist a tangent cone), but it has no "double points", see definition in the second section. Moreover, the fact that the domain of dependence of the solution $u(\cdot , \cdot)$ of $(L)$  is finite implies that the hypothesis of compactness of the support of $w$ is actually unessential, see {\it e.g.} footnote at the end of \textsection{5.3} in \cite{john}.

\end{document}